\documentclass[10pt]{article}

\usepackage{latexsym,amssymb,amsmath,amsthm,paralist}
\usepackage[all]{xy}
\def\mapr#1{\stackrel{#1}{\longrightarrow}}

\def\surjd#1{\lower4pt\hbox{$\downarrow$}\kern-5.65pt\Big\downarrow\rlap {$\vcenter{\hbox{$\scriptstyle{{#1}}$}}$}}
\newcommand{\R}{{\mathbb R}}

\newcommand{\Spec}{\text{\it Spec}}
\newcommand{\Pic}{\text{\it Pic}}
\newcommand{\Br}{\text{\it Br}}
\newcommand{\cor}{\text{\it cor}}
\newcommand{\Q}{{\mathbb Q}}
\newcommand{\Z}{{\mathbb Z}}

\newcommand{\G}{{\mathbb G}}
\renewcommand{\O}{{\mathcal{O}}}

\newcommand{\Cl}{\text{\it Cl}}
\newcommand{\p}{{\mathfrak p}}

\newcommand{\sm}{{\smallsetminus}}
\newcommand{\et}{\text{\it et}}
\newcommand{\fl}{\text{\it fl}\,}
\newcommand{\cd}{\text{\it cd}\:}

\newcommand{\Gal}{\text{\it Gal}}

\font\emas = cmsy10 scaled\magstep2
\newcommand{\freeproductmed}{\mathop{\lower.2mm\hbox{\emas \symbol{3}}}\limits}
\newcommand{\lang}{\longrightarrow}

\renewcommand{\min}{\text{\rm min}}
 \newcommand{\ressum}{\mathop{\hbox{${\displaystyle\bigoplus}'$}}\limits}

\newcommand{\nr}{\mathit{nr}}

 1

\newtheoremstyle{alex}
  {}
  {}
  {\sl}
  {}
  {\bf}
  {.}
  {.5em}
  {}
\newtheoremstyle{alexdef}
  {}
  {}
  {\sl}
  {}
  {\bf}
  {.}
  {.5em}
  {}
\theoremstyle{alex}
 
\newtheorem{theorem}{Theorem}[section]

\newtheorem{lemma}[theorem]{Lemma}
\newtheorem{proposition}[theorem]{Proposition}

\theoremstyle{alexdef}

\title{\bf\boldmath On the $K(\pi,1)$-property for rings of integers in the mixed case}
\author{by Alexander Schmidt}

\begin{document}
\maketitle

\begin{abstract}  We investigate the Galois group $G_S(p)$ of the maximal $p$-extension unramified outside a finite set $S$ of primes of a number field in the (mixed) case, when there are primes dividing~$p$ inside and outside $S$. We show that the cohomology of $G_S(p)$ is `often' isomorphic to the \'{e}tale cohomology of the scheme $\Spec(\O_k \sm S)$, in particular, $G_S(p)$ is of cohomological dimension~$2$ then. We deduce this from the results in our previous paper~\cite{kpi1}, which mainly dealt with the tame case.
\end{abstract}

\section{Introduction}
Let $Y$ be a connected locally noetherian scheme and let $p$ be a prime number. We denote the \'{e}tale fundamental group of $Y$ by $\pi_1(Y)$ and its maximal pro-$p$ factor group by $\pi_1(Y)(p)$. The Hochschild-Serre spectral sequence induces natural homomorphisms
\[
\phi_{i}: H^i(\pi_1^\et(Y)(p), \Z/p\Z) \longrightarrow H^i_\et(Y, \Z/p\Z), \ i\geq 0,
\]
and we call $Y$ a \lq$K(\pi,1)$ for $p$\rq\ if all $\phi_i$ are isomorphisms; see \cite{kpi1} Proposition~2.1 for equivalent conditions. See \cite{wi2} for a purely Galois cohomological approach to the $K(\pi,1)$-property. Our main result is the following

\begin{theorem}\label{thma}
Let $k$ be a number field and let $p$ be a prime number. Assume that $k$ does not contain a primitive $p$-th root of unity and that the class number of~$k$ is prime to $p$. Then the following holds:

Let $S$ be a finite set of primes of $k$  and let $T$ be a set of primes of $k$ of Dirichlet density $\delta(T)=1$.  Then there exists a finite subset $T_1\subset T$ such that
$\Spec(\O_k)\sm (S\cup T_1)$ is a $K(\pi,1)$ for $p$.
\end{theorem}

\noindent
{\bf Remarks.} 1. If $S$ contains the set $S_p$ of primes dividing~$p$, then Theorem~\ref{thma} holds with $T_1=\varnothing$ and even without the condition $\zeta_p \notin k$ and $\Cl(k)(p)=0$, see \cite{kpi1}, Proposition~2.3.  In the tame case $S\cap S_p=\varnothing$, the statement of Theorem~\ref{thma} is the main result of~\cite{kpi1}. Here we provide the extension to the \lq mixed\rq\  case $\varnothing \varsubsetneq S\cap S_p \varsubsetneq S_p$.

\smallskip\noindent
2. For a given number field $k$, all but finitely many prime numbers~$p$ satisfy the condition of Theorem~\ref{thma}. We conjecture that  Theorem~\ref{thma} holds without the restricting assumption on  $p$.

\bigskip
Let $S$ be a finite set of places of a number field $k$. Let $k_S(p)$ be the maximal $p$-extension of $k$ unramified outside~$S$ and put  $G_S(p)=\Gal(k_S(p)|k)$. If $S_\R$ denotes the set of real places of $k$, then $G_{S\cup S_\R}(p)\cong \pi_1(\Spec(\O_k)\sm S)(p)$ (we have $G_S(p)=G_{S\cup S_\R}(p)$ if $p$ is odd or $k$ is totally imaginary).  The following Theorem~\ref{extra} sharpens Theorem~\ref{thma}.

\begin{theorem}\label{extra}
The set $T_1\subset T$ in Theorem~\ref{thma} may be chosen such that

\smallskip
\begin{compactitem}
\item[\rm (i)] $T_1$ consists of primes $\p$ of degree $1$  with $N(\p)\equiv 1 \bmod p$,
\item[\rm (ii)] $(k_{S\cup T_1}(p))_\p=k_\p(p)$ for all primes $\p\in S\cup T_1$.
\end{compactitem}
\end{theorem}

\noindent
Note that Theorem~\ref{extra} provides nontrivial information even in the case $S\supset S_p$, where assertion (ii) was only known when $k$ contains a primitive $p$-th root of unity (Kuz'min's theorem, see \cite{kuz} or \cite{NSW}, 10.6.4 or \cite{NSW2}, 10.8.4, respectively) and for certain CM fields (by a result of Mukhamedov, see \cite{muk} or \cite{NSW}, X \S6 exercise or \cite{NSW2}, X \S8 exercise, respectively).

\smallskip
By  Theorem~\ref{dualmod} below, Theorem~\ref{extra} provides many examples of $G_S(p)$ being a duality group.  If $\zeta_p\notin k$, this is interesting even in the case that $S\supset S_p$, where examples of $G_S(p)$ being a duality group were previously known only for real abelian fields and for certain CM-fields (see \cite{NSW}, 10.7.15 and \cite{NSW2}, 10.9.15, respectively, and the remark following there).

\bigskip
Previous results in the mixed case had been achieved by K.~Wingberg \cite{wing}, Ch.\ Maire \cite{maire} and D.\ Vogel \cite{Vo}. Though not explicitly visible in this paper, the present progress in the subject was only possible due to the results on mild pro-$p$ groups obtained by J.\ Labute in \cite{La}.

\bigskip
I would like to thank K.~Wingberg for pointing out that the proof of Proposition~8.1 in my paper \cite{kpi1} did not use the assumption that the sets $S$ and $S'$ are disjoint from $S_p$. This was the key observation for the present paper.
The main part of this text was written while I was a guest at the Department of Mathematical Sciences of Tokyo University and of the Research Institute for Mathematical Sciences in Kyoto. I want to thank these institutions for their kind hospitality.

\section{Proof of Theorems~\ref{thma} and \ref{extra}}

We start with the observation that the proofs of Proposition~8.2 and Corollary~8.3 in \cite{kpi1} did not use the assumption that the sets $S$ and $S'$ are disjoint from $S_p$. Therefore, with the same proof (which we repeat for the convenience of the reader) as in loc.\ cit., we obtain

\begin{proposition} \label{enlarge} Let $k$ be a number field and let $p$ be a prime number. Assume $k$ to be totally imaginary if $p=2$. Put $X=\Spec(\O_k)$ and
let $S\subset S'$ be finite sets of primes of $k$.  Assume that $X\sm S$ is a $K(\pi,1)$ for $p$ and that $G_S(p)\neq 1$. Further assume that each $\p\in S'\sm S$ does not split completely in $k_S(p)$. Then the following hold.

\smallskip
\begin{compactitem}
\item[{\rm (i)}] $X\sm S'$ is a $K(\pi,1)$ for $p$.
\item[{\rm (ii)}] $k_{S'}(p)_\p=k_\p(p)$ for all $\p \in S'\sm S$.
\end{compactitem}

\smallskip\noindent
Furthermore, the arithmetic form of Riemann's existence theorem holds, i.e., setting $K=k_S(p)$, the natural homomorphism
\[
\freeproductmed_{{\mathfrak p} \in S'\backslash S(K)} T(K_\p(p)|K_\p) \longrightarrow \Gal(k_{S'}(p)|K)
\]
is an isomorphism. Here $T(K_\p(p)|K_\p)$ is the inertia group and $\freeproductmed$ denotes the free pro-$p$-product of a bundle of pro-$p$-groups, cf.\ \cite{NSW}, Ch.\,IV,\,\S3. In particular, $\Gal(k_{S'}(p)|k_S(p))$ is a free pro-$p$-group.
\end{proposition}

\begin{proof} The $K(\pi,1)$-property implies
\[
H^i(G_S(p),\Z/p\Z) \cong H^i_\et(X\sm S,\Z/p\Z)=0 \text{ for } i\geq 4,
\]
hence $\cd\ G_S(p)\leq 3$. Let $\p \in S'\sm S$. Since $\p$ does not split completely in $k_S(p)$ and since $\cd\ G_S(p)< \infty$, the decomposition group of $\p$ in $k_S(p)|k$ is a non-trivial and torsion-free quotient of $\Z_p\cong \Gal(k_\p^{nr}(p)|k_\p)$. Therefore $k_S(p)_\p$ is the maximal unramified $p$-extension of $k_\p$. We denote the normalization of an integral normal scheme $Y$ in an algebraic extension $L$ of its function field by $Y_L$.
Then $(X\sm S)_{k_S(p)}$ is the universal pro-$p$ covering of $X\sm S$.  We consider the \'{e}tale excision sequence for
the pair $((X \sm S)_{k_S(p)}, (X \sm S')_{k_S(p)})$. By assumption, $X\sm S$ is a $K(\pi,1)$ for $p$, hence   $H^i_{\et}((X \sm S)_{k_S},\Z/p\Z)=0$ for $i \geq 1$ by \cite{kpi1}, Proposition~2.1. Omitting the coefficients $\Z/p\Z$ from the notation, this implies isomorphisms
\[
H^i_{\et}\big((X \sm S')_{k_S(p)}\big) \stackrel{\sim}{\to} \ressum_{{\mathfrak p} \in S'\sm S (k_S(p))} H^{i+1}_{\mathfrak p}\big(((X \sm S)_{k_S})_{\mathfrak p}\big)
\]
for $i\geq 1$. Here (and in variants also below) we use the notational convention
\[
\ressum_{{\mathfrak p} \in S'\sm S (k_S(p))} H^{i+1}_{\mathfrak p}\big(((X \sm S)_{k_S(p)})_{\mathfrak p}\big)
:=\varinjlim_{K\subset k_S(p)} \bigoplus_{{\mathfrak p} \in S'\sm S (K)} H^{i+1}_{\mathfrak p}\big(((X \sm S)_{K})_{\mathfrak p}\big),
\]
where $K$ runs through the finite extensions of $k$ inside $k_S(p)$. As $k_S(p)$ realizes the maximal unramified $p$-extension of $k_\p$ for all $\p\in S'\sm S$, the schemes $((X \sm S)_{k_S(p)})_{\mathfrak p}$, $\p\in S'\sm S (k_S(p))$, have trivial cohomology with values in $\Z/p\Z$ and we obtain isomorphisms
\[
H^i((k_S(p))_\p)\stackrel{\sim}{\to} H^{i+1}_{\mathfrak p}\big(((X \sm S)_{k_S(p)})_{\mathfrak p}\big)
\]
for $i\geq 1$.
These groups vanish for $i\geq 2$. This implies
\[
H^i_{\et}((X\sm S')_{k_S(p)})=0
\]
for $i\geq 2$. Since the scheme $(X\sm S')_{k_{S'}(p)}$ is the universal pro-$p$ covering of
$(X\sm S')_{k_{S}(p)}$,  the Hochschild-Serre spectral sequence yields an inclusion
\[
H^2(\Gal(k_{S'}(p)|k_S(p))) \hookrightarrow H^2_{\et}((X\sm S')_{k_{S}(p)})=0.
\]
Hence $\Gal(k_{S'}(p)|k_S(p))$ is a free pro-$p$-group and
\[
H^1(\Gal(k_{S'}(p)|k_S(p))) \stackrel{\sim}{\to} H^1_{\et}((X\sm S')_{k_S(p)})
\cong \ressum_{{\mathfrak p}\in S'\sm S(k_S(p))} H^1(k_S(p)_{\mathfrak p}).
\]
We set $K=k_S(p)$ and consider the natural homomorphism
\[
\phi: \freeproductmed_{{\mathfrak p} \in S'\backslash S(K)} T(K_\p(p)|K_\p) \longrightarrow \Gal(k_{S'}(p)|K).
\]
By the calculation of the cohomology of a free product (\cite{NSW}, 4.3.10 and 4.1.4), $\phi$ is a homomorphism between free pro-$p$-groups which induces an isomorphism on mod~$p$ cohomology. Therefore $\phi$ is an isomorphism. In particular, $k_{S'}(p)_\p=k_\p(p)$ for all $\p\in S'\sm S$.
Using that $\Gal(k_{S'}(p)|k_S(p))$ is free, the Hochschild-Serre spectral sequence
\[
E_2^{ij}=H^i\big(\Gal(k_{S'}(p)|k_S(p)), H^j_\et((X\sm S')_{k_{S'}(p)})\big) \Rightarrow H^{i+j}_\et((X\sm S')_{k_S(p)})
\]
induces an isomorphism
\[
0=H^2_\et((X\sm S')_{k_{S}(p)})\mapr{\sim} H^2_\et((X\sm S')_{k_{S'}(p)})^{\Gal(k_{S'}|k_S)}.
\]
Hence $H^2_\et((X\sm S')_{k_{S'}(p)})=0$, since $\Gal(k_{S'}(p)|k_S(p))$ is a pro-$p$-group.  Now \cite{kpi1}, Proposition~2.1  implies that $X\sm S'$ is a $K(\pi,1)$ for $p$.
\end{proof}

In order to prove Theorem~\ref{thma}, we first provide the following lemma. For an extension field $K|k$ and a set of primes $T$ of $k$, we write $T(K)$ for the set of prolongations of primes in $T$ to $K$ and $\delta_K(T)$ for the Dirichlet density of the set of primes $T(K)$ of $K$.

\begin{lemma}\label{nonsplit}
Let $k$ be a number field, $p$ a prime number and $S$ a finite set of nonarchimedean primes of $k$. Let $T$ be a set of primes of $k$ with $\delta_{k(\mu_p)}(T)=1$.  Then there exists a finite subset $T_0\subset T$  such that all primes $\p\in S$ do not split completely in the extension $k_{T_0}(p)|k$.
\end{lemma}

\begin{proof}
By \cite{NSW}, 9.2.2\,(ii) or \cite{NSW2}, 9.2.3\,(ii), respectively, the restriction map
\[
H^1(G_{T\cup S\cup S_p\cup S_\R}(p), \Z/p\Z) \longrightarrow \prod_{\p\in S\cup S_p\cup S_\R} H^1(k_\p,\Z/p\Z)
\]
is surjective. A class in $\alpha\in H^1(G_{T\cup S\cup S_p\cup S_\R}(p),\Z/p\Z)$ which restricts to an unramified class $\alpha_\p \in H^1_{\nr}(k_\p,\Z/p\Z)$ for all $\p \in S\cup S_p\cup S_\R$ is contained in $H^1(G_{T}(p),\Z/p\Z)$.  Therefore the image of the composite map
\[
H^1(G_T(p),\Z/p\Z) \hookrightarrow H^1(G_{T\cup S\cup S_p\cup S_\R}(p),\Z/p\Z) \rightarrow \prod_{\p\in S} H^1(k_\p,\Z/p\Z)
\]
contains the subgroup $\prod_{\p\in S} H^1_{nr}(k_\p,\Z/p\Z)$. As this group is finite, it is already contained in the image of $H^1(G_{T_0}(p),\Z/p\Z)$ for some finite subset $T_0\subset T$.  We conclude that no prime in $S$ splits completely in the maximal elementary abelian $p$-extension of $k$ unramified outside~$T_0$.
\end{proof}

\begin{proof}[Proof of Theorems~\ref{thma} and \ref{extra}]
As $p\neq 2$, we may ignore archimedean primes. Furthermore, we may remove the primes in $S\cup S_p$ and all primes of degree greater than~$1$ from $T$. In addition, we remove all primes $\p$ with $N(\p)\not\equiv 1 \bmod p$ from $T$.
After these changes, we still have $\delta_{k(\mu_p)}(T)=1$.

By Lemma~\ref{nonsplit}, we find a finite subset $T_0\subset T$ such that no prime in $S$ splits completely in $k_{T_0}(p)|k$.  Put $X=\Spec(\O_k)$.  By \cite{kpi1}, Theorem 6.2, applied to $T_0$ and $T\sm T_0$, we find a finite subset $T_2\subset T\sm T_0$ such that $X\sm (T_0\cup  T_2)$ is a $K(\pi,1)$ for $p$. Then Proposition~\ref{enlarge} applied to $T_0\cup T_2 \subset S\cup T_0\cup T_2$, shows that also $X\sm (S\cup T_0\cup T_2)$ is a $K(\pi,1)$ for $p$. Now put $T_1= T_0\cup T_2\subset T$.

It remains to show Theorem~\ref{extra}. Assertion (i) holds by construction of $T_1$. By \cite{kpi1}, Lemma 4.1, also $X\sm (S'\cup T_1)$ is a $K(\pi,1)$ for $p$.  By \cite{kpi1}, Theorem~3, the field $k_{T_1}(p)$ realizes $k_\p(p)$ for $\p\in T_1$, showing (ii) for these primes. Finally, assertion (ii) for $\p\in S$ follows from Proposition~\ref{enlarge}.
\end{proof}

\section{Duality}
We start by investigating the relation between the $K(\pi,1)$-property and the universal norms of global units.

Let us first remove redundant primes from $S$: If $\p \nmid p$ is a prime with $\zeta_p\notin k_\p$, then every $p$-extension of the local field~$k_\p$ is unra\-mi\-fied (see \cite{NSW}, 7.5.1 or \cite{NSW2}, 7.5.9, respectively).  Therefore primes $\p \notin  S_p$ with $N(\p) \not \equiv 1 \bmod p$  cannot ramify in a $p$-extension.
Removing all these redundant primes from $S$, we obtain a subset $S_{\min} \subset S$, which has the property that $G_S(p)=G_{S_{\min}}(p)$. Furthermore, by \cite{kpi1}, Lemma~4.1, $X\sm S$ is a $K(\pi,1)$ for $p$ if and only if $X\sm S_\min$ is a $K(\pi,1)$ for $p$.

\begin{theorem} \label{thmb} Let $k$ be a number field and let $p$ be a prime number. Assume that $k$ is totally imaginary if $p=2$.  Let $S$ be a finite  set of nonarchimedean primes of\/ $k$. Then any two of the following conditions {\rm (a) -- (c)} imply the third.

\smallskip\noindent
\begin{compactitem}
\item[\rm (a)] $\Spec({\cal O}_k)\sm S$ is a $K(\pi,1)$ for $p$. \smallskip
\item[\rm (b)] $\varprojlim_{K\subset k_S(p)} \O_K^\times \otimes \Z_p=0$. \smallskip
\item[\rm (c)] $(k_{S}(p))_\p=k_\p(p)$ for all primes $\p\in S_\min$.
\end{compactitem}

\smallskip\noindent
The limit in {\rm (b)} runs through all finite extensions $K$ of $k$ inside $k_S(p)$.
If {\rm (a)--(c)} hold, then also
\[
\varprojlim_{K\subset k_S(p)} \O_{K,S_\min}^\times \otimes \Z_p=0.
\]
\end{theorem}

\noindent
{\bf Remarks:} 1. Assume that $\zeta_p \in k$ and $S\supset S_p$. Then (a) holds and condition (b) holds for $p>2$ if $\# S > r_2+2$ (see \cite{NSW2}, Remark 2 after 10.9.3). In the case  $k=\Q(\zeta_p)$, $S=S_p$, condition (b) holds if and only if $p$ is an irregular prime number.

\noindent
2. Assume that $S\cap S_p=\varnothing$  and $S_\min\neq \varnothing$. If condition (a) holds, then either $G_S(p)=1$ (which only happens in very special situations, see \cite{kpi1}, Proposition 7.4) or (b) holds by \cite{kpi1}, Theorem~3 (or by Proposition~\ref{fulllocal} below).

\begin{proof}[Proof of Theorem~\ref{thmb}] We may assume $S=S_\min$ in the proof.
Let $K$ run through the finite extensions of $k$ in $k_S(p)$ and put $X_K=\Spec(\O_K)$. Applying the topological Nakayama-Lemma (\cite{NSW}, 5.2.18) to the compact $\Z_p$-module $\varprojlim \O_K^\times \otimes \Z_p$, we see that condition (b) is equivalent to

\medskip
\begin{compactitem}
\item[\rm (b)']\quad  $\varprojlim_{K\subset k_S(p)} \O_K^\times/p=0$.
\end{compactitem}

\medskip\noindent
Furthermore, by \cite{kpi1}, Proposition~2.1, condition  (a) is equivalent to

\medskip
\begin{compactitem}
\item[\rm (a)'] \quad $\varinjlim_{K\subset k_S(p)} H^i_\et((X\sm S)_K, \Z/p\Z)=0$ for $i\geq 1$.
\end{compactitem}

\medskip\noindent
Condition (a)' always holds for $i=1$, $i\geq 4$, and it holds for $i=3$ provided that $G_S(p)$ is infinite or  $S$ is nonempty or $\zeta_p \notin k$ (see \cite{kpi1}, Lemma 3.7).
The flat Kummer sequence $0\to \mu_p \to \G_m \stackrel{\cdot p}{\to} \G_m \to 0$ induces exact sequences
\[
0 \lang \O_K^\times /p \lang H^1_\fl(X_K,\mu_p) \lang \null_p \Pic(X) \to 0
\]
for all $K$. As the field $k_S(p)$ does not have nontrivial unramified $p$-extensions, class field theory implies
\[
\varprojlim_{K\subset k_S(p)} \null_p\Pic(X_K) \subset \varprojlim_{K\subset k_S(p)} \Pic(X_K)\otimes \Z_p=0.
\]
As we assumed  $k$ to be totally imaginary if $p=2$, the flat duality theorem of Artin-Mazur (\cite{Mi}, III Corollary 3.2) induces natural isomorphisms
\[
 H^2_\et(X_K,\Z/p\Z)= H^2_\fl(X_K,\Z/p\Z) \cong H^1_\fl(X_K, \mu_p)^\vee.
\]
  We conclude that
\[
\varinjlim_{K\subset k_S(p)} H^2_\et(X_K,\Z/p\Z)\cong \big(\varprojlim_{K\subset k_S(p)} \O_K^\times /p \ \big)^\vee. \leqno (*)
\]
We first show the equivalence of (a) and (b) in the case $S=\varnothing$. If (a)' holds, then $(\ast)$ shows (b)'. If (b) holds, then $\zeta_\p\notin k$ or $G_S(p)$ is infinite. Hence we obtain (a)' for $i=3$. Furthemore, (b)' implies (a)' for $i=2$ by $(\ast)$. This finishes the proof of the case $S=\varnothing$.

Now we assume that $S\neq \varnothing$. For ${\mathfrak p} \in S(K)$, a standard calculation of local cohomology shows that
\[
H^i_{\mathfrak p}(X_K, \Z/p\Z)\cong \left\{
\begin{array}{cc}
0& \text{for } i\leq 1,\\
H^1(K_\p,\Z/p\Z)/H^1_\nr(K_\p,\Z/p\Z)& \text{for } i=2,\\
H^2(K_\p,\Z/p\Z)& \text{for } i=3.\\
0 &\text{for } i\geq 4.
\end{array}
 \right.
\]
For $\p\in S=S_\min$,  every proper Galois subextension of  $k_\p(p)|k_\p$ admits ramified $p$-extensions. Hence condition (c) is equivalent to

\medskip
\begin{compactitem}
\item[\rm (c)'] \quad $\varinjlim_{K\subset k_S(p)} \bigoplus_{\p \in S(K)} H^i_\p(X_K, \Z/p\Z)=0$ for all $i$,
\end{compactitem}

\medskip\noindent
and to
\medskip
\begin{compactitem}
\item[\rm (c)''] \quad $\varinjlim_{K\subset k_S(p)} \bigoplus_{\p \in S(K)} H^2_\p(X_K, \Z/p\Z)=0$.
\end{compactitem}

\medskip\noindent
Consider the direct limit over all $K$ of  the excision sequences
\[
\cdots \to  \bigoplus_{\p\in S(K)} H^i_\p(X_K,\Z/p\Z) \to H^i_\et(X_K,\Z/p\Z) \to H^i_\et((X\sm S)_K,\Z/p\Z) \to \cdots.
\]
Assume that (a)' holds, i.e.\ the right hand terms vanish in the limit for $i\geq 1$.
Then $(\ast)$ shows that (b)' is equivalent to (c)''.

Now assume that (b) and (c) hold. As above, (b) implies the vanishing of the middle term for $i=2,3$ in the limit.
Condition (c)' then shows (a)'.

We have proven that any two of the conditions (a)--(c) imply the third.

\medskip
Finally, assume that (a)--(c) hold. Tensoring the exact sequences (cf.\ \cite{NSW}, 10.3.11 or \cite{NSW2}, 10.3.12, respectively)
\[
0 \to \O_K^\times \to \O_{K,S}^\times \to \bigoplus_{\p\in S(K)} (K_\p^\times/U_\p) \to \Pic(X_K) \to \Pic((X\sm S)_K) \to 0
\]
by (the flat $\Z$-algebra) $\Z_p$, we obtain exact sequences of finitely generated, hence compact, $\Z_p$-modules. Passing to the projective limit over the finite extensions $K$ of $k$ inside $k_S(p)$ and using $\varprojlim \Pic(X_K)\otimes \Z_p =0$, we obtain the exact sequence
\[
0 \to \varprojlim_{K\subset k_S(p)} \O_K^\times \otimes \Z_p \to \varprojlim_{K\subset k_S(p)} \O_{K,S}^\times \otimes \Z_p \to \varprojlim_{K\subset k_S(p)} \bigoplus_{\p\in S(K)} (K_\p^\times /U_\p) \otimes \Z_p \to 0.
\]
Condition (c) and local class field theory imply the vanishing of the right hand limit.
Therefore (b) implies the vanishing of the projective limit in the middle.
\end{proof}

If $G_S(p)\neq 1$ and condition (a) of Theorem~\ref{thma} holds, then the failure in condition (c) can only come from primes dividing $p$. This follows from the next

\begin{proposition} \label{fulllocal}
Let $k$ be a number field and let $p$ be a prime number. Assume that $k$ is totally imaginary if $p=2$.  Let $S$ be a finite  set of nonarchimedean primes of\/ $k$.  If $\Spec(\O_k)\sm S$ is a $K(\pi,1)$ for $p$ and $G_S(p)\neq 1$, then every prime $\p\in S$ with $\zeta_p\in k_\p$ has an infinite inertia group in $G_S(p)$. Moreover,  we have
\[
k_S(p)_\p=k_\p(p)
\]
for all  $\p\in S_\min \sm S_p$.
\end{proposition}

\begin{proof} We may assume $S=S_\min$.
Suppose $\p\in S$ with $\zeta_p\in k_\p$ does not ramify in $k_S(p)|k$. Setting $S'=S\sm \{\p\}$, we have $k_{S'}(p)=k_S(p)$, in particular,
\[
H^1_\et (X\sm S',\Z/p\Z) \stackrel{\sim}{\lang} H^1_\et (X\sm S,\Z/p\Z).
\]
In the following, we omit the coefficients $\Z/p\Z$ from the notation. Using the vanishing of $H^3_\et(X\sm S)$, the \'{e}tale excision sequence yields a commutative exact diagram
\[
\xymatrix@=.7cm{&H^2(G_{S'}(p))\ar[r]^\sim\ar@{^{(}->}[d]&H^2(G_S(p))\ar[d]^\wr\\
H^2_\p(X)\ar@{^{(}->}[r]&H^2_\et(X\sm S')\ar[r]^\alpha&H^2_\et(X\sm S)\ar[r]&H^3_\p(X)\ar@{->>}[r]&H^3_\et(X\sm S').}
\]
Hence $\alpha$ is split-surjective and $\Z/p\Z\cong H^3_\p(X)\stackrel{\sim}{\to} H^3_\et(X\sm S')$. This implies $S'=\varnothing$, hence $S=\{\p\}$, and $\zeta_p\in k$. The same applies to every finite extension of $k$ in $k_S(p)$, hence $\p$ is inert in $k_S(p)=k_\varnothing(p)$. This implies that the natural homomorphism
\[
\Gal(k_\p^\nr(p)|k_\p) \lang G_\varnothing(k)(p)
\]
is surjective. Therefore $G_S(p)=G_\varnothing(p)$ is abelian, hence finite by class field theory. Since this group has finite cohomological dimension by the $K(\pi,1)$-property, it is trivial, in contradiction to our assumptions.

This shows that all $\p\in S$ with $\zeta_p\in k_\p$ ramify in $k_S(p)$. As this applies to every finite extension of $k$ inside $k_S(p)$, the inertia groups must be infinite. For $\p\in S_\min\sm S_p$ this implies  $k_S(p)_\p=k_\p(p)$.
\end{proof}

\begin{theorem} \label{dualmod} Let $k$ be a number field and let $p$ be a prime number. Assume that $k$ is totally imaginary if $p=2$.  Let $S$ be a finite nonempty set of nonarchimedean primes of\/ $k$.
Assume that conditions {\rm (a)--(c)} of Theorem~\ref{thmb} hold and that $\zeta_p\in k_\p$ for all $\p\in S$. Then $G_S(p)$ is a pro-$p$ duality group of dimension~$2$.
\end{theorem}

\begin{proof} Condition (a) implies $H^3(G_{S}(p),\Z/p\Z) \stackrel{\sim}{\to} H^3_\et (X\sm S,\Z/p\Z)=0$. Hence $\cd\ G_{S}(p)\leq 2$. On the other hand, by (c),  the group $G_{S}(p)$ contains $\Gal(k_\p(p)|k_\p)$ as a subgroup for all $\p\in S$. As $\zeta_p\in k_\p$ for $\p\in S$, these local groups have cohomological dimension~$2$, hence so does  $G_{S}(p)$.

In order to show that $G_S(p)$ is a duality group, we have to show that
\[
D_i(G_S(p),\Z/p\Z): = \varinjlim_{\substack{U\subset G_S(p)\\ \cor^\vee}} H^i(U, \Z/p\Z)^\vee
\]
vanish for $i=0,1$, where $U$ runs through the open subgroups of $G_S(p)$ and the transition maps are the duals of the corestriction homomorphisms; see \cite{NSW}, 3.4.6. The vanishing of $D_0$ is obvious, as $G_S(p)$ is infinite.
Using (a), we therefore have to show that
\[
\varinjlim_{K\subset k_S(p)} H^1((X\sm S)_K,\Z/p\Z)^\vee=0.
\]
We put $ X=\Spec({\cal O}_k)$ and  denote the embedding by $j: (X\sm S)_K \to  X_K$. By the flat duality theorem of Artin-Mazur, we have natural isomorphisms
\[
H^1((X\sm S)_K,\Z/p\Z)^\vee\cong H^2_{\fl,c}((X\sm S)_K, \mu_p)= H^2_\fl(X_K, j_! \mu_p).
\]
The excision sequence together with a straightforward calculation of local cohomology groups shows an exact sequence
\[
\bigoplus_{\p \in S(K)} K_\p^\times/K_\p^{\times p} \to H^2_\fl(X_K, j_! \mu_p) \to H^2_\fl((X\sm S)_K, \mu_p).\leqno (\ast)
\]
As $\zeta_p\in k_\p$ and $k_S(p)_\p=k_\p(p)$ for $\p\in S$ by assumption, the left hand term of $(\ast)$ vanishes when passing to the limit over all $K$.  We use the Kummer sequence to obtain an exact sequence
\[
\Pic((X\sm S)_K)/p \lang H^2_\fl((X\sm S)_K, \mu_p) \lang \null_p \Br((X\sm S)_K).\leqno (\ast\ast)
\]
The left hand term of $(\ast\ast)$ vanishes in the limit by the principal ideal theorem. The Hasse principle for the Brauer group induces an injection
\[
\null_p \Br((X\sm S)_K) \hookrightarrow \bigoplus_{\p\in S(K)} \null_p\Br(K_\p).
\]
As $k_S(p)$ realizes the maximal unramified $p$-extension of $k_\p$ for $\p\in S$, the limit of the middle term in $(\ast\ast)$, and hence also the limit of then middle term in $(\ast)$ vanishes. This shows that $G_S(p)$ is a duality group of dimension $2$.
\end{proof}

\noindent
{\bf Remark:} The dualizing module can be calculated to
\[
D \cong \text{\rm tor}_{p}\big(C_S(k_S(p)\big),
\]
i.e.\ $D$ is isomorphic to the $p$-torsion subgroup in the $S$-id\`{e}le class group of $k_S(p)$. The proof is the same as in (\cite{circular}, Proof of Thm.\ 5.2), where we dealt with the tame case.

\vskip2cm

\noindent \footnotesize{Alexander Schmidt, NWF I - Mathematik, Universit\"{a}t Regensburg, D-93040
Regensburg, Deutschland. email: alexander.schmidt@mathematik.uni-regensburg.de}

\enddocument